\documentclass[12pt]{amsart}

\newtheorem{theorem}{Theorem}[section]
\newtheorem{lemma}[theorem]{Lemma}
\newtheorem{corollary}[theorem]{Corollary}
\newtheorem{proposition}[theorem]{Proposition}
\theoremstyle{definition}   
\newtheorem{definition}{Definition}
\newtheorem{example}[theorem]{Example}

\theoremstyle{remark}

\numberwithin{equation}{section}

\usepackage[utf8]{inputenc}
\usepackage[english]{babel}
\usepackage{times}
\usepackage{enumerate}
\usepackage{tfrupee}
\usepackage{tikz}
\usetikzlibrary{chains,fit}
\usetikzlibrary{shapes,snakes}
\usetikzlibrary{graphs}
\usepackage{graphicx,adjustbox}
\usepackage{hyperref}
\usepackage{amsmath,amssymb}
\usepackage{amscd}
\usepackage[all,cmtip]{xy}
\usepackage{calrsfs} 
\usepackage{tabularx}
\usepackage{amsthm}
\usepackage{mathalfa,mathrsfs}
\usepackage[all]{xy}
\usepackage{amscd}
\usepackage{graphicx}
\usepackage[all]{xy}
\begin{document}

\title{Colon structure of 
associated primes of monomial ideals}

\author{Siddhi Balu Ambhore 
\and 
Indranath Sengupta}

\date{}

\address{\small \rm  Discipline of Mathematics, IIT Gandhinagar, Palaj, Gandhinagar, 
Gujarat 382355, INDIA.}
\email{siddhi.ambhore@iitgn.ac.in}
\thanks{This paper is the outcome of the research work carried out by the 
first author at IIT Gandhinagar, first as a student of M.Sc., then as a 
research intern in \textit{e-SRIP} 2020, and finally as a 
\textit{Sabarmati Bridge Fellow}. This work has been done under the 
supervision of the second author. Both the authors gratefully acknowledge 
the encouragement and support received from IIT Gandhinagar.
}

\address{\small \rm  Discipline of Mathematics, IIT Gandhinagar, Palaj, Gandhinagar, 
Gujarat 382355, INDIA.}
\email{indranathsg@iitgn.ac.in}
\thanks{The second author is the corresponding author; supported by the 
MATRICS research grant MTR/2018/000420, sponsored by the SERB, Government of India.}

\date{}


\allowdisplaybreaks

\begin{abstract}
We find an explicit expression of the associated primes of 
monomial ideals as a colon by an element $v$, using the unique irredundant 
irreducible decomposition whose irreducible components are monomial
ideals (Theorem 3.1). An algorithm to compute $v$ is given using Macaulay2
(Section 7).
For squarefree monomial ideals the problem is related to the combinatorics
of the underlying clutter or graph (Proposition 4.3). For ideals
of Borel type the monomial $f$ takes a simpler form (Proposition 5.2). The
authors classify when $f$ is unique (Proposition 6.2).
\end{abstract}

\maketitle

\section{Introduction}
An important result in commutative algebra states that 
an associated prime ideal $P$ of a monomial ideal $I$ can be expressed 
as $P = (I \colon v)$, for a monomial $v$; see \cite{hh}. 
This paper provides an explicit form of the monomial $v$, for 
any given monomial ideal $I$ and an associated prime ideal $P$, using the  irredundant irreducible decomposition of $I$. For example, 
for a monomial ideal $I = \langle x_1^4,x_2^7,x_3^5,x_1^3x_4^2,x_2^4x_4^2,x_3x_4^2,x_4^5,x_4^2x_8^2,x_1x_8^8 \rangle \subset S = \mathbb{Q}[x_1, \dots, x_8]$ and $P = \langle x_1, x_2,x_3,x_4 \rangle$, we can have $v = x_2^6x_3^4x_4x_5^5x_6^5x_7^2x_8^{13}$ which gives us $P = (I \colon v)$. We have illustarted this example in section 7, where we have given an \emph{Macaulay2} algorithm to compute such $v$'s using theorem 3.1. We also study the uniqueness of such a $v$, and prove conditions for its uniqueness. We treat the cases of 
squarefree monomial ideals and ideals of Borel type separately to 
show that $v$ appears in a much simpler form. Another interesting 
class of monomial ideals which has been investigated closely is the class of 
edge ideals associated with clutters and simple finite graphs. R.H. Villarreal 
has drawn our attention to the notion of $\mbox{v}$-number of an ideal 
$I$ (see \cite{vn}), which is nicely related to the monomial $v$ 
that is our object of study. This work cannot be extended directly 
for binomial 
ideals for the reason that a binomial irreducible decomposition for 
binomial ideals may not always exist; see \cite{ex}. However, it 
is known that if  $I \subset S$ is a graded ideal and $P \in \mbox{Ass}(I)$, 
then there is ahomogeneous polynomial $f$ in $S$ such that $(I: f) = P$. There is an algorithm
to compute $f$ [\cite{R1}, Example 5.1, Procedure A.1]. In particular, this algorithm
applies to monomial and binomial ideals.

\section{Preliminaries}
Let $S  = K[x_1, \dots, x_n]$ be a polynomial ring in $n$ variable over 
the field $K$. A product $x_1^{a_1}\cdots x_n^{a_n}$, with 
$a_i \in \mathbb{Z}_{\geq 0}$, is called a \textit{monomial}. 
If $u = x_1^{a_1}\cdots x_n^{a_n}$ is a monomial, then we write 
$u = \mathbf{x^a}$, with 
$\mathbf{a} = (a_1,\dots,a_n) \in (\mathbb{Z}_{\geq 0})^n$. 
The $n$-tuple $\mathbf{a} = (a_1,\dots,a_n)$ is called the 
\textit{index} of monomial $u$ and it is denoted as $\mbox{index}(u)$. 
Let us define $\mbox{floor}(u) := \{i \mid x_{i} \mid u\}$. The 
monomial $\displaystyle{\prod_{i\in\mbox{floor}(u)}}x_i$ 
is called the \textit{base} of $u$, denoted by $\mbox{base}(u)$. 
Let $u$ be a monomial. Let $\nu_i(u)$ denote the highest power of $x_i$ 
in $u$, if $x_i$ divides $u$.

\begin{definition} [\cite{hh}]{\rm
A presentation of an ideal $I$ as $I = \cap_{i=1}^m Q_i$ of ideals $Q_i$ 
is called \textit{irredundant} if none of the ideals $Q_i$ can be omitted 
in this presentation.
}
\end{definition}

\begin{proposition} [\cite{hh}]
Each monomial ideal has a unique minimal monomial set
of generators. More precisely, let $G$ denote the set of monomials in $I$ which are minimal with respect to divisibility. Then $G$ is the unique minimal set of monomial generators.
\end{proposition}
We denote the unique minimal set of monomial generators of the monomial ideal $I$ by $G(I)$.

\begin{theorem} [\cite{hh}] \label{1}
Let $I \subset S = K[x_1,\dots,x_n]$ be a monomial ideal. Then $I = \cap_{i=1}^m Q_i$, where each $Q_i$ is generated by pure powers of the variables. In other words, each $Q_i$ is of the form $\langle x^{a_1}_{i_1},\dots,x^{a_k}_{i_k} \rangle$. Moreover, an irredundant presentation of this form is unique. 
\end{theorem}

\begin{definition} [\cite{iid}]
Let $I$ be an ideal of $S$. A set of irreducible ideals 
$\{Q_i\}_{i=1,\dots,m}$ is called an 
\textit{irredundant irreducible decomposition} (in short IID) 
of $I$, if $I = \cap^m_{i=1} Q_i$ and $Q_i \nsupseteq \cap_{j \neq i} Q_j$, 
for all $i$.
\end{definition}

\begin{proposition} [\cite{hh}] \label{2}
The irreducible ideal $\langle x_{i_1}^{a_1},\dots,x_{i_k}^{a_k} \rangle $ is $\langle x_{i_1},\dots,x_{i_k}\rangle $-primary.
\end{proposition}

\begin{corollary} [\cite{hh}] \label{a}
Let \(I \subset S\) be a monomial ideal, and let \(P \in \text{Ass}(I)\). Then there exists a monomial \(v\) such that 
\(P= (I \colon v)\).
\end{corollary}

\begin{corollary} [\cite{hh}] \label{f}
If $I$ is a squarefree monomial ideal, then, $$I = \cap_{P \in Min(I)} P,$$ where $Min(I)$ is the set of minimal prime ideals of $I$.
\end{corollary}

\begin{theorem} [\cite{am}] \label{b}
Let $A$ be a ring. Let $\mathfrak{a} \subset A$ be a decomposable ideal and let $\mathfrak{a} = \cap_{i = 1}^n \mathfrak{q}_i$ be a minimal primary decomposition of $\mathfrak{a}$. Let $\mathfrak{p}_i = r(\mathfrak{q}_i) \; (1 \leqslant i \leqslant n)$. Then the $\mathfrak{p}_i$ are precisely the prime ideals which occur in the set of ideals $r(\mathfrak{a} \colon x) \; (x\in A)$, and hence are independent of the particular decomposition of $\mathfrak{a}$.
\end{theorem}

\section{The colon structure}
\begin{theorem} \label{3}
Let $I \subset S$ be a monomial ideal, with irredundant irreducible 
decomposition $I = \cap_{i=1}^r Q_i$ and 
$P = \langle x_{i_1}, \dots, x_{i_k} \rangle \in Ass(I)$. 
Let $Q = \langle x_{i_1}^{a_1},\dots,x_{i_k}^{a_k} \rangle$ 
be a $P$-primary component of $I$ in the irredundant 
irreducible decomposition of $I$. Let $\{s_{1}, \ldots, s_{n-k}\}= [n] \backslash \{i_1, \dots , i_k\}$. For any choice of 
$b_j \geqslant \max\{\nu_j(u) \mid u \in G(I)\}$, \, $j \in [n-k]$, the monomial 
\begin{equation}\label{4}
v= x_{i_1}^{a_1 - 1} x_{i_2}^{a_2 - 1} \dots x_{i_k}^{a_k - 1} x_{s_1}^{b_1} \dots x_{s_{n-k}}^{b_{n-k}}
\end{equation}
always satisfies $P = (I \colon v)$.
\end{theorem}

\begin{proof}
We need to prove that $P = (I \colon v)$, where $v$ is as given in equation 
\eqref{4}. We prove that $P \subset (I \colon v)$ and $(I \colon v) \subset P$.
\medskip

Let $p \in P$. We need to prove that $p \in (I \colon v)$. We know that $(\cap I_m) \colon J = \cap (I_m \colon J)$, which gives us $(I \colon v) = \cap_{i = 1}^r (Q_i : v)$. Thus, we will prove that $p \in \cap_{i = 1}^r (Q_i \colon v)$ i.e., $p \in (Q_i \colon v)$,  for all $i \in [r]$ i.e., $pv \in Q_i$,  for all $i \in [r]$. Since, $p \in P$ implies $p = x_{i_1} h_1 + \dots + x_{i_k} h_k$, 
where $h_i \in S$, for all $1 \leqslant i \leqslant k$, then 
$$pv = (x_{i_1} h_1 + \dots + x_{i_k} h_k)v = \sum_{j=1}^k x_{i_j} h_j v.$$ 
We need to show that $pv \in Q_i$, for all $i \in [r]$.
To prove this we make cases based on the set of generators of $Q_i$'s.
\medskip

\noindent\textbf{Case (i).} Let the set of generators of $Q_i$ contain  
$x_{s_j}^{t_j}$, for some $t_j \in \mathbb{N}$ and 
$j \in [n]\backslash \{i_1, \dots , i_k \}$. 
Notice that $t_j \leqslant b_j$, as $b_j \geqslant$ max$\{\nu_j(u) \mid u \in G(I)\}$. Therefore, $x_{s_j}^{t_j} | v$, 
which implies $x_{s_j}^{t_j} | pv$, and hence $pv \in Q_i$.
\medskip

\noindent\textbf{Case (ii).} Let the set of generators of $Q_i$ does not contain 
any power of $x_{s_j}$. This implies that the 
set of generators of $Q_i$ contains pure powers of $\{x_{i_m}\}_{m \in F}$, 
where $F \subset [k]$. If $Q_i = Q$, then we can clearly see that 
$x_{i_j} h_j v \in Q$, as $x_{i_j}^{a_j} | x_{i_j} h_j v$, for all $j \in [k]$. 
Hence, $\sum_{j=1}^k x_{i_j} h_j v = pv \in Q$. If $Q_i \neq Q$, 
then we must notice that atleast one of $x_{i_m}$ must have a power 
strictly less than $a_m$, otherwise, if all the powers of $x_{i_m}$ 
are greater than or equal to $a_m$, then we get that 
$Q_i \cap Q = Q_i$, implying that we can omit $Q$ from the 
irredundant irreducible decomposition of $I$, which is a contradiction. 
Hence, the exists $j \in F$, such that $x_{i_j}^{c_j} \in G(Q_i)$, 
with $c_j < a_j$. Thus, we get $x_{i_j}^{c_j} | v$. Therefore, 
we get that $pv \in Q_i$.
\medskip

Therefore, we get that $pv \in Q_i,$ for all $i \in [r]$, which implies $pv \in \cap_{i=1}^r Q_i = I$. Hence, we get $p \in (I \colon v)$. Thus, 
$p \in P$ implies that $p \in (I \colon v)$, i.e., $P \subset (I \colon v)$.
\medskip

Let $f \in (I \colon v)$, we want to prove that $f \in P$. 
Since $f \in (I \colon v)$ implies $fv \in I$, we can 
always write $f = \sum_{u \in supp(f)} a_u u$, where 
$a_u \in K$. Let $u \in supp(f)$. We assume that no $x_{i_j}$ 
divides $u$. Since, $I$ is a monomial ideal and we have 
$fv \in I$, therefore $uv \in I$, i.e., 
$$u \cdot (x_{i_1}^{a_1 - 1} x_{i_2}^{a_2 - 1} \dots x_{i_k}^{a_k - 1} x_{s_1}^{b_1} \dots x_{s_{n-k}}^{b_{n-k}}) \in I.$$ 
Since, $I = \cap_{i=1}^r Q_i$, we get $I \subset Q_i$, for all $i \in [r]$. 
In particular, $I \subset Q$. Thus, we get that 
$u \cdot (x_{i_1}^{a_1 - 1} x_{i_2}^{a_2 - 1} \dots x_{i_k}^{a_k - 1} x_{s_1}^{b_1} \dots x_{s_{n-k}}^{b_{n-k}}) \in Q = \langle x_{i_1}^{a_1},\dots,x_{i_k}^{a_k} \rangle$, which is a contradiction, as power of $x_{i_j}$ in $uv$ is strictly 
less than $a_j$ for all $j \in [k]$. Hence, our assumption is wrong. 
Therefore, there exists $j \in [k]$, such that $x_{i_j} | u$. This 
implies that $u \in \langle x_{i_1}, \dots ,x_{i_k} \rangle = P$, which 
implies that $u \in P$, for all $u \in supp(f)$, which gives $f \in P$. 
Hence, $f \in (I \colon v)$ implies $f \in P$. 
Therefore, $(I \colon v) \subset P$.
\end{proof}

\begin{example}
Let $S = K[x_1,x_2,x_3]$. Let $I = \langle x_1x_2^3x_3^3 ,\; x_1^3x_2^3x_3 ,\; x_1^3x_2x_3^3\rangle$ or 
$I = \langle x_1^2x_2^4x_3^5 ,\; x_2^2x_1^4x_3^5 ,\; x_1^2x_3^4x_2^5 ,\; x_2^2x_3^4x_1^5 , \; x_3^2x_1^4x_2^5, \; x_3^2x_2^4x_1^5 \rangle$. 
Note that, the indices of all the generators have a particular form 
for both the examples. It is $\{(1,3,3), (3,3,1), (3,1,3)\}$ for the 
first example and $$\{(2,4,5), (4,2,5), (2,4,5), (5,2,4), (4,5,2), (5,4,2)\}$$ for 
the second example. 
We try to rewrite the exponents to have a better description of $v$.
\end{example}

Let $S= K[x_1,\dots, x_n]$ be the polynomial ring over the field $K$. 
For a fixed $k\in\{1,\ldots , n\}$ and 
$a_1 \leqslant a_2 \leqslant \cdots \leqslant a_k$ 
in $\mathbb{Z}_{\geq 0}$, we define the monomial ideal $I$ as follows:
$$I = \langle x_{i_1}^{a_1} x_{i_2}^{a_2} \cdots x_{i_k}^{a_k} \colon 1 \leqslant i_p \leqslant n ; \, i_p \neq i_q \text{ whenever } p \neq q ; 1 \leqslant p,\, q \leqslant k \rangle .$$
Suppose there are $r$ different $a_i$'s. Then we can write $a_i$'s as
\begin{eqnarray*}
a_1 = \cdots = a_{k_2 -1} & < & a_{k_2} = \cdots = a_{k_3 -1}\\
{} & < & a_{k_3} = \cdots = a_{k_4 -1}\\
{} & < & a_{k_4} = \cdots = a_{k_5 -1}\\
{} & < & \vdots\\
{} & < & a_{k_r} = \cdots = a_k.
\end{eqnarray*} 

\noindent For the sake of consistency, we denote $a_1 = a_{k_1}$. A 
primary decomposition (similar to the irredundant irreducible 
decomposition) of $I$ is 
$$I = \bigcap_{j = 1}^r \bigg( \bigcap_{\substack{i_p \in [n], \; i_p \neq i_q, \forall p \neq q, \\ p, q \in [n - k + k_j]}} \langle x_{i_1}^{a_{k_j}} , \dots, x_{i_{n - k + k_j}}^{a_{k_j}} \rangle \bigg),$$ 
with the set of associated prime ideals of $I$ given by
$$
Ass(I) = \{\langle x_{i_1}, \dots, x_{i_{n - k + k_j}} \rangle \bigr| \; i_p \in [n], i_p \neq i_q, \forall p \neq q; p , q \in [n - k + k_j], j \in [r] \}.
$$

\begin{corollary} \label{5}
Let $S$ be the polynomial ring and $I$ be the monomial ideal defined as above. 
Then, for any $P = \langle x_{i_1}, \dots, x_{i_{n - k + k_j}} \rangle \in Ass(I)$, where $j \in [r]$, we have $\langle x_{i_1}^{a_{k_j}} , \dots, x_{i_{n - k + k_j}}^{a_{k_j}} \rangle$ as a $P$-primary component in the primary decomposition 
of $I$ and we can write $P = (I \colon v)$, for 
$$v = x_{i_1}^{a_{k_j} -1} x_{i_2}^{a_{k_j} -1} \cdots x_{i_{n - k + k_j}}^{a_{k_j} -1} x_{s_1}^{b_1} \dots x_{s_{k-k_j}}^{b_{k-k_j}},$$ 
where $s_t \in [n] \backslash \{i_1, \dots , i_{n - k + k_j}\}$, 
all $s_t$'s are distinct and $b_t \geqslant a_{k_j + t} $, 
for all $t \in [k- k_j]$.
\end{corollary}

\begin{proof}
Let $P \in Ass(I)$, then $P = \langle x_{i_1}, x_{i_2}, \dots, x_{i_{n - k + k_j}} \rangle$ for some $j \in [r]$, and we have 
$Q = \langle x_{i_1}^{a_{k_j}} , \dots, x_{i_{n - k + k_j}}^{a_{k_j}} \rangle$ as a $P$-primary component.
We need to show that $P  = (I \colon v)$, where $v = x_{i_1}^{a_{k_j -1}} x_{i_2}^{a_{k_j -1}} \cdots x_{i_{n - k + k_j}}^{a_{k_j -1}} x_{s_1}^{b_1} \dots x_{s_{k-k_j}}^{b_{k-k_j}}$, such that 
$s_t \in [n] \backslash \{i_1, \dots , i_{n - k + k_j}\}$, all $s_t$'s are distinct and $b_t \geqslant a_{k_j + t} $ for 
$t \in [k- k_j]$.
We prove that $P \subset (I \colon v)$ and $(I \colon v) \subset P$.
\medskip

Let $p \in P$, then $p = \sum_{t = 1}^{n - k +k_j} x_{i_t} f_t$, where $f_t \in S$. We need to show $pv \in I$. We have
$$pv = \Bigr(\sum_{t = 1}^{n - k +k_j} x_{i_t} f_t \Bigr)v = \sum_{t = 1}^{n - k +k_j} (x_{i_t} f_t v)$$
It is enough to prove that $x_{i_t} f_t v \in I$, 
for all $t \in [n - k +k_j]$.
\medskip

Consider, 
$$x_{i_t} f_t v = x_{i_1}^{a_{k_j -1}} \cdots x_{i_{t-1}}^{a_{k_j -1}} x_{i_{t + 1}}^{a_{k_j -1}} \cdots x_{i_{n - k + k_j}}^{a_{k_j -1}} x_{i_{t}}^{a_{k_j}} x_{s_1}^{b_1} \cdots x_{s_{k-k_j}}^{b_{k-k_j}} f_t.$$
Now we choose ($k_j -1$) variables out of ($n - k + k_j - 1$) variables, which is possible since $n \geqslant k$. Thus, 
we can have a monomial $$u = x_{p_1}^{a_1} \cdots x_{p_{k_j -1}}^{a_{k_j -1}} x_{i_{t}}^{a_{k_j}} x_{s_1}^{a_{k_j +1}} \cdots x_{s_{k-k_j}}^{a_k} \in I,$$ where $p_m \in \{i_1, \dots, i_{t - 1}, i_{t+1}, \dots, i_{n - k +k_j} \}$ and $m \in [k_j -1]$. Clearly, $u | x_{i_t} f_t v$, which gives us that $x_{i_t} f_t v \in I$. Since 
$t \in [r]$ is arbitrary, we get that $x_{i_t} f_t v \in I$, for all $t \in [r]$. Thus we get that $pv \in I$, i.e., $P \subset (I \colon v)$.
\medskip

Let $f \in (I \colon v)$. We need to show that $f \in P$. We 
have $fv \in I$, therefore if we take $u \in supp(f)$ then 
$uv \in I$ as $I$ is an monomial ideal. Suppose, none of 
$x_{i_j}$ divides $u$, for all $j \in [n- k+ k_j]$. We have 
$uv \in I$ implies that $uv \in \langle x_{i_1}^{a_{k_j}} , \dots, x_{i_{n - k + k_j}}^{a_{k_j}} \rangle$, and therefore 
$$u x_{i_1}^{a_{k_j -1}} x_{i_2}^{a_{k_j -1}} \cdots x_{i_{n - k + k_j}}^{a_{k_j -1}} x_{s_1}^{b_1} \dots x_{s_{k-k_j}}^{b_{k-k_j}} \in \langle x_{i_1}^{a_{k_j}} , \dots, x_{i_{n - k + k_j}}^{a_{k_j}} \rangle,$$ 
which is contradiction. Thus, there must exist 
$j \in [n- k+ k_j]$, such that $x_{i_j} | u$. 
This gives that 
$u \in \langle x_{i_1}, x_{i_2}, \dots, x_{i_{n - k + k_j}} \rangle 
= P$. Since $u \in supp(f)$ is arbitrary, we get $u \in P$, 
for all $u \in supp(f)$. Hence, $f \in P$. Thus, we get 
$(I \colon v) \subset P$.
\end{proof}

\begin{proposition} \label{c}
Let $I \subset S$ be a monomial ideal and $P = \langle x_{i_1}, \dots, x_{i_k} \rangle \in Ass(I)$. If $v = x_{i_1}^{a_1} \cdots x_{i_k}^{a_k} x_{i_{k+1}}^{a_{k+1}} \cdots x_{i_n}^{a_n}$ be such that $P = (I \colon v)$, then, there exists $Q$ in the irredundant irreducible decomposition of $I$ such that $Q = \langle x_{i_1}^{a_1 + 1}, \dots, x_{i_k}^{a_k + 1} \rangle$.
\end{proposition}

\begin{proof}
Let $I = \cap_{i =1}^r Q_i$ be the irredundant irreducible decomposition of $I$. 
We are given that 
$P = (I \colon v) = (( \cap_{i =1}^r Q_i ) \colon v) 
= \cap_{i =1}^r (Q_i \colon v)$. Since, $P$ is an irreducible 
ideal, there exists $Q$ in an irredundant irreducible decomposition of $I$, 
such that $P = Q \colon v$. Since, $Q$ is an irreducible ideal, we have 
$Q = \langle x_t^{b_t} | t \in F \subset [n] \rangle$. By theorem \ref{b} 
applied to $Q$, we get $F = \{i_1, \dots, i_k \}$. Thus, we get 
$Q = \langle x_{i_1}^{b_{i_1}}, \dots, x_{i_k}^{b_{i_k}} \rangle$. 
For simplicity we write $b_{i_j} = b_j$. Thus, we have 
$Q = \langle x_{i_1}^{b_1}, \dots, x_{i_k}^{b_k} \rangle$. 
Now we show that $b_j = a_j + 1$ holds for all $j \in [k]$. 
Assume that there exist $j \in [k]$ such that $b_j \neq a_j +1$. Then, 
$b_j < a_j + 1$ or $b_j > a_j +1$. If $b_j < a_j +1$, then 
$b_j \leqslant a_j$. This gives $x_{i_j}^{b_j} | x_{i_j}^{a_j}$, 
which implies 
$x_{i_j}^{b_j} | v$. Thus $v \in Q$, which implies that 
$Q \colon v = S$, which is not true. Hence, our assumption  
$b_j < a_j + 1$ is wrong. Therefore, we must have $b_j \nless a_j + 1$. 
Now, suppose that $b_j > a_j + 1$. Since $P = (Q \colon v)$, 
we have $x_{i_j}v \in Q$, therefore $x_{i_t}^{b_t} | x_{i_j}v$. If $t \neq j$, then $x_{i_t}^{b_t} | x_{i_j}v$ implies $x_{i_t}^{b_t} | v$. Again, this gives 
$v \in Q$, which implies $(Q \colon v) = S$, which is not true. 
Hence $t = j$. Thus, we get $x_{i_j}^{b_j} | x_{i_j}v$, that is 
$x_{i_j}^{b_j} | x_{i_1}^{a_1} \cdots x_{i_j}^{a_j +1} \cdots x_{i_k}^{a_k}$. 
This implies $b_j \leqslant a_j +1$, which contradicts the assumption that $b_j > a_j + 1$. Hence, $b_j \ngtr a_j + 1$. 
Therefore, $\nexists j \in [k]$, such that 
$b_j \neq a_j +1$. Therefore, $b_j = a_j + 1$ holds for all $j \in [k]$.
\end{proof}

We can use \ref{c} to prove the following result for squarefree monomial ideals.

\begin{proposition} \label{d}
Let $I$ be a squarefree monomial ideal in $S$ and let $P = \langle x_{i_1}, \dots, x_{i_k} \rangle$ belong to $Ass(I)$. Let $P  = (I \colon v)$, for some monomial $v$. 
Then, the powers of $x_{i_j}$'s in $v$ are zero, i.e., $v$ is of the form 
$v= x_{s_1}^{b_1} \cdots x_{s_{n-k}}^{b_{n-k}}$, where $s_j \in [n] \backslash \{i_1, \dots, i_k\}$ and $s_p \neq s_q$ whenever $p \neq q$.
\end{proposition}

\begin{proof}
Let 
$v = x_{i_1}^{c_1}\cdots x_{i_k}^{c_k} x_{s_1}^{b_1} \cdots x_{s_{n-k}}^{b_{n-k}}$, such that $P = I \colon v$. If $c_i = 0$, for all $i \in [k]$, then we are done. Suppose, there exist $j \in [k]$ such that $c_j \neq 0$. This implies $c_j \geqslant 1$. Then by Proposition \ref{c}, there exist $Q = \langle x_{i_1}^{c_1 +1} , \dots, x_{i_j}^{c_j +1}, \dots, x_{i_k}^{c_k +1} \rangle$ in an irredundant irreducible decomposition of $I$. Note that $c_j \geqslant 1$ implies that $c_j + 1 \geqslant 2$, which is not true as $Q$ is an ideal in an irredundant irreducible decomposition 
of squarefree monomial ideal $I$. Hence, our assumption is wrong. Therefore, 
$c_i = 0$ for all $i \in [k]$.
\end{proof}

\section{Ideals corresponding to graphs and Clutters}
We set the notations and definitions for clutters as given in \cite{vn}. Let 
$S = K[t_1, \dots ,t_s] = \oplus_{d=0}^{\infty} S_d$ be the standard 
graded polynomial ring over a field $K$. Let $\mathcal{C}$ be a clutter 
with vertex set $V(\mathcal{C}) = \{t_1, \dots,t_s\}$, that is, 
$\mathcal{C}$ is a family of subsets of $V(\mathcal{C})$, called edges, 
none of which is included in another. The set of edges of $\mathcal{C}$ 
is denoted by $E(\mathcal{C})$. The edge ideal of $\mathcal{C}$, denoted $I(\mathcal{C})$, is the ideal of $S$ generated by all squarefree monomials 
$t_e = \prod_{t_i \in e} t_i$, such that $e \in E(\mathcal{C})$. A subset 
$A$ of the vertex set $V(\mathcal{C})$ is called a 
stable set of vertices of $\mathcal{C}$ if $A$ does not contain any 
edge of $\mathcal{C}$. The neighbour set of $A$, 
denoted $N_{\mathcal{C}}(A)$, is the set of all vertices $t_i$ such that 
$\{t_i\} \cup A$ contains an edge of $\mathcal{C}$. The set $\mathcal{F_C}$ 
denotes the family of all maximal stable sets of $\mathcal{C}$ and 
$\mathcal{A_C}$ denotes the family of all stable sets $A$ of $\mathcal{C}$ 
whose neighbor set $N_{\mathcal{C}}(A)$ is a minimal vertex cover of 
$\mathcal{C}$.

\begin{lemma} [\cite{vn}] \label{e}
Let $I = I(\mathcal{C})$ be the edge ideal of a clutter $\mathcal{C}$. The following hold:
\begin{enumerate}[(i)]
\item If $A \in \mathcal{A_C}$ and $t_A = \prod_{t_i \in A} t_i$, then $(I \colon t_A) = (N_{\mathcal{C}}(A))$.
\item If $A$ is stable and $N_{\mathcal{C}}(A)$ is a vertex cover, then $N_{\mathcal{C}}(A)$ is a minimal vertex cover.
\item If $(I \colon f) = \mathfrak{p}$ for some $f \in S_d$ and some $\mathfrak{p} \in Ass(I)$, then there is $A \in \mathcal{A_C}$ with $|A| \leqslant d$ such that $\mathfrak{p} = (N_{\mathcal{C}}(A))$ and $(I \colon t_A) = (N_{\mathcal{C}}(A))$.
\item If $A \in \mathcal{F_C}$, then $N_{\mathcal{C}}(A) = V(\mathcal{C}) \backslash A$ and $(I \colon t_A) = (N_{\mathcal{C}}(A))$.
\end{enumerate}
\end{lemma}

\begin{proposition} [\cite{ma}] \label{g}
Let $C$ be a set of vertices of a clutter $\mathcal{C}$ and let $\mathfrak{p}$ be the face ideal of $R = K[x_1,\dots,x_n]$ generated by $C$. The following are equivalent:
\begin{enumerate}
\item $C$ is a minimal vertex cover of $\mathcal{C}$.
\item $\mathfrak{p}$ is a minimal prime of $I(\mathcal{C})$.
\item $V(\mathcal{C}) \backslash C$ is a maximal face of $\Delta_{\mathcal{C}}$, where $\Delta_{\mathcal{C}}$ is the simplicial complex whose faces are the independent vertex sets of $\mathcal{C}$.
\end{enumerate}
\end{proposition}

\begin{proposition} \label{10}
Let $I = I(\mathcal{C})$ be the edge ideal of a clutter $\mathcal{C}$. Let  
$\mathfrak{p} \in Ass(I)$ corresponding to $A \in \mathcal{F_C}$, with 
$\mathfrak{p} = \langle N_{\mathcal{C}}(A)\rangle $. Then $t_A$ forms the base for $v$, 
such that $\mathfrak{p} = (I \colon v)$. 
\end{proposition}

\begin{proof}
We know that edge ideal is a squarefree monomial ideal. Then, by Corollary \ref{f}, 
$\mathfrak{p} \in Ass(I)$ implies $\mathfrak{p}$ is minimal prime ideal of $I$. 
Thus, by Proposition \ref{g}, we get a corresponding minimal vertex cover of 
$\mathcal{C}$ for $\mathfrak{p}$. The proof follows by Lemma \ref{e} and 
Proposition \ref{d}.
\end{proof}

The simple finite graphs are the examples of clutters. Hence, proposition \ref{10} also gives us the base for the $v$ corresponding to $P \in Ass(I(\mathcal{G}))$, where $I(\mathcal{G})$ is a graph ideal (or edge ideal) corresponding to a simple finite graph $\mathcal{G}$.

\section{Ideals of Borel type}
\begin{definition} [\cite{hh}]{\rm
A monomial ideal $I \subset S = K[x_1,\dots,x_n]$ is of 
\textit{Borel type} if $I \colon x_i^{\infty} = I \colon \langle x_1,\dots,x_i \rangle^{\infty}$, for $i = 1,\dots, n$.
}
\end{definition}

\begin{theorem} [\cite{hh}] \label{7} 
Let $I \subset S$ be a monomial ideal. The following conditions are equivalent:
\begin{enumerate}
\item $I$ is of Borel type.

\item For each monomial $u\in I$ and all integers $i,j,s$, 
with $1 \leqslant j < i \leqslant n$ and $s > 0$, such that $x_i^s |u$, 
there exists an integer $t \geqslant 0$ such that $x_j^t(u/x_i^s) \in I$.

\item For each monomial $u \in I$ and all integers $i,j$ with 
$1 \leqslant j < i \leqslant n$, there exists an integer 
$t \geqslant 0$ such that $x_j^t(u/x_i^{\nu_i(u)}) \in I$.

\item If $P \in Ass(I)$, then $P = \langle x_1,\dots,x_j \rangle$ for some $j$.
\end{enumerate}
\end{theorem}

\begin{proposition}
Let $I \subset S$ be monomial ideal of Borel type with irredundant irreducible decomposition $I = \cap_{i=1}^r Q_i$. Then, for $P \in Ass(I)$, $P$ will be of the form $P = \langle x_1,\dots,x_k \rangle$. Let $Q = \langle x_1^{a_1},\dots,x_k^{a_k} \rangle$ be the $P$-primary component of $I$ in the irredundant irreducible decomposition of $I$. Then, for $k < n$, we can write $P = (I \colon v)$, where
\begin{equation}\label{8}
v= x_1^{a_1 - 1} x_2^{a_2 - 1} \dots x_k^{a_k - 1} x_{k+1}^{b_{k+1}}
\end{equation}
and $b_{k+1} \geqslant$ max$\{\nu_{k+1}(u) \mid u \in G(I)\}$.
\end{proposition}

\begin{proof}
We need to prove that $P = (I \colon v)$, where $v$ is as given in equation 
\eqref{8}. We now show that $P \subset (I \colon v)$ and $(I \colon v) \subset P$.
\medskip

Let $p \in P$. We need to prove that $p \in I \colon v$, i.e., $pv \in I$. Since, $p \in P$ implies 
$p = x_1 h_1 + \dots + x_k h_k$, where $h_i \in S, \forall 1 \leqslant i \leqslant k$, 
we get 
$$pv = (x_1 h_1 + \dots + x_k h_k)v = \sum_{j=1}^k x_j h_j v.$$
We need to show that $pv \in I$ i.e. $pv \in Q_i$, for all  $i \in [r]$.
To prove this we make cases based on the set of generators of $Q_i$'s.
\medskip

\noindent\textbf{Case (i).} Let us assume that the set of generators of 
$Q_i$ contains $x_s^{t_s}$, for some $t_s \in \mathbb{N}$ and $s > k$. 
As $Q_i$ contains $x_s^t$ as one of its generators with $s>k$, hence 
from Theorem \ref{7}, we get that $Q_i$ is of the form $\langle x_1^{t_1}, \dots, x_k^{t_k}, x_{k+1}^{t_{k+1}}, \dots, x_s^{t_s}, \dots \rangle$. 
Note that $t_{k+1} \leqslant b_{k+1}$, as $b_{k+1}\geqslant$ max$\{\nu_{k+1}(u) \mid u \in G(I)\}$. Therefore, $x_{k+1}^{t_{k+1}} | v$, 
which implies that $x_{k+1}^{t_{k+1}} | pv$. Hence, $pv \in Q_i$.
\medskip

\noindent\textbf{Case (ii).} Let us assume that the 
set of generators of $Q_i$ does not contain any power of $x_s$ with $s>k$. 
This implies that set of generator of $Q_i$ contains pure powers of 
$\{x_m\}_{m \in F}$, where $F=\{1,\dots,l\}$ for some $l \leqslant k$. 
If $Q_i = Q$, then we can clearly see that $x_j h_j v \in Q$, as 
$x_j^{a_j} | x_j h_j v$ for all $j \in [k]$. Hence, 
$\sum_{j=1}^k x_j h_j v = pv \in Q$. If $Q_i \neq Q$, 
then atleast one of $x_m$ must have 
power strictly less than $a_m$, otherwise if all the powers 
of $x_m$ are greater than or equal to $a_m$, then we 
get that $Q_i \cap Q = Q_i$. This implies that we can omit 
$Q$ from the irredundant irreducible decomposition of $I$, which is contradiction. Hence, there exists $j \in F$ such that $x_j^{c_j} \in G(Q_i)$, 
with $c_j < a_j$. We get $x_j^{c_j} | v$, which implies 
that $x_j^{c_j} | pv$. 
This gives us $pv \in Q_i$, and hence $pv \in \cap_{i=1}^r Q_i = I$. 
Hence, we get $p \in (I \colon v)$. 
\medskip

Now, let $f \in (I \colon v)$. We want to prove that $f \in P$. 
Since $f \in (I \colon v)$, this implies that 
$fv \in I$. We can always write $f = \sum_{u \in \mbox{supp}(f)} a_u u$, 
where $a_u \in K$. Let $u \in \mbox{supp}(f)$ and let us assume that none of 
$x_j$ divides $u$. $I$ is a monomial ideal and we have 
$fv \in I$, therefore we get $uv \in I$, i.e., 
$u \cdot (x_1^{a_1 - 1} x_2^{a_2 - 1} \dots x_k^{a_k - 1} x_{k+1}^{b_{k+1}}) \in I$. Since $I = \cap_{i=1}^r Q_i$, we get $I \subset Q_i$, for all $i \in [r]$. 
In particular, $I \subset Q$. Thus, we get that 
$u \cdot (x_1^{a_1 - 1} x_2^{a_2 - 1} \dots x_k^{a_k - 1} x_{k+1}^{b_{k+1}}) \in Q = \langle x_1^{a_1},\dots,x_k^{a_k} \rangle$, which is a contradiction, 
as power of $x_j$ in $uv$ is strictly less than $a_j$ for all $j \in [k]$. 
Hence, our assumption is wrong. Therefore, $\exists j \in [k]$ 
such that $x_j | u$. This implies that $u \in \langle x_1, \dots ,x_k \rangle = P$, 
i.e., $u \in P$,  for all $u \in \mbox{supp}(f)$. Therefore, $f \in P$.
\end{proof}

Note that if $k = n$, i.e., if $P = \langle x_1, \dots, x_n \rangle$, 
then for the $P$-primary component $Q = \langle x_1^{a_1},\dots,x_n^{a_n} \rangle$   
we have $v = x_1^{a_1 -1}\cdots x_n^{a_n -1}$, and we do not have to take care of any further variables in $v$.

\section{Uniqueness of $v$}
It is easy to see that $v$ need not be unique, 
as if we have $x_{s_j}$ appearing with the power $b_j$, 
we can have infinitely many $v$ with power of $x_{s_j}$ greater than 
$b_j$ and those should work as well. We can ask for 
uniqueness of $v$ in terms of the base and powers 
of variables dividing $v$. However, we have observed that 
such uniqueness is not possible in the general case. 
Let us first see an example to show that the uniqueness of 
base for $v$ fails. The following also shows that the minimum unique power for the variables of $v$ cannot be obtained for the general case. We have used \textit{Macaulay2} for understanding such examples.

\begin{example}
Let $S = \mathbb{Q}[x_1, \dots, x_6]$ and $I = \langle x_1x_3^5,x_2^4x_5^3,x_2^4x_4^4,x_1^5x_4^2,x_1x_6^8 \rangle$ be an 
monomial ideal in $S$. Then, $P = \langle x_1, x_2 \rangle \in \mbox{Ass}(I)$. 
For $P$, we can 
have $v = x_3^5x_4^4 \text{ or } x_3^5x_5^3 \text{ or } x_6^8x_4^4 \text{ or } x_6^8x_5^3$.
\end{example}
\medskip

The example shows that we cannot have the uniqueness of $v$, 
in terms of base or the powers of the variables occurring in $v$. 
This example also illustrates that we cannot find the minimum powers 
of the variables in $v$, such that all the powers greater than that 
power gives us $P = (I \colon v)$ and no power less than that gives 
$P = (I \colon v)$. Take $P = \langle x_1, x_2 \rangle \in Ass(I)$. 
For $P$, we can have $v = x_3^5x_4^4$, such that $P = (I \colon v)$. 
Here, power of $x_5$ is $0$, as we can write 
$v = x_3^5x_4^4 = x_3^5x_4^4x_5^0$. Moreover, 
we can see that $2 > 0$ as power of $x_5$ does not work if we consider 
the monomial $x_3^5x_5^2$. We do not get the expected output, as can 
be seen by simple computation with \textit{Macaulay2}. 
However, the following corollary gives us 
information about the nature of $P$ and $I$, when $v$ is unique.

\begin{proposition}
Let $I \subset S = K[x_1, \dots, x_n]$ be a monomial ideal and 
$P \in Ass(I)$. The monomial $v$ is unique for $P$ such that 
$P = (I \colon v)$ if and only if $P = \langle x_1, \dots, x_n \rangle$ 
and $P$ has unique $P$-primary component in irredundant irreducible 
decomposition of $I$.
\end{proposition}

\begin{proof}
($\Rightarrow$) Suppose $P = \langle x_{i_1}, \dots, x_{i_k} \rangle$, 
for some $1 \leqslant k < n$. Let $Q = \langle x_{i_1}^{a_1}, \dots, x_{i_k}^{a_k} \rangle$ be a $P$-primary component in irredundant irreducible 
decomposition of $I$. Then, by Theorem \ref{3}, we have 
$$v = x_{i_1}^{a_1 - 1} x_{i_2}^{a_2 - 1} \dots x_{i_k}^{a_k - 1} x_{s_1}^{b_1} \dots x_{s_{n-k}}^{b_{n-k}},$$ 
and $b_j$'s are as defined in Theorem \ref{3}. We can have infinite choices 
for $b_j$'s, which give us infinitely many $v$'s, which is not possible. 
Hence, our supposition is wrong. Therefore, we get 
$P = \langle x_1, \dots, x_n \rangle$. Now, suppose $P = \langle x_1, \dots, x_n \rangle$ does not have a unique $P$-primary component in the 
irredundant irreducible decomposition of $I$. It is enough to consider the 
case when $P$ has two $P$-primary components in the irredundant irreducible decomposition, say $Q_1 = \langle x_1^{\alpha_1}, \dots, x_n^{\alpha_n} \rangle$ and $Q_2 = \langle x_1^{\beta_1}, \dots, x_n^{\beta_n} \rangle$. Then, $\alpha_i \neq \beta_i$ for some $1 \leqslant i \leqslant n$. By Theorem \ref{3}, we can have $v = x_1^{\alpha_1 - 1} \cdots x_n^{\alpha_n - 1}$ and $v' = x_1^{\beta_1 - 1} \cdots x_n^{\beta_n -1}$, corresponding to $Q_1$ and $Q_2$ respectively, 
such that $P = (I \colon v)$ and $P = (I \colon v')$. Note that 
$v \neq v'$ as $\alpha_i \neq \beta_i$ for some $1 \leqslant i \leqslant n$. 
Hence, we get two different monomials $v$ and $v'$ such that $P = (I \colon v)$ 
and $P = (I \colon v')$, which contradicts the given condition. Thus, $P$ must 
have a unique $P$-primary component in irredundant irreducible decomposition.
\medskip

($\Leftarrow$) Let $P = \langle x_1, \dots, x_n \rangle$ and let 
$Q = \langle x_1^{a_1}, \dots , x_n^{a_n} \rangle$ be the unique 
$P$-primary component in the irredundant irreducible decomposition 
$I = \cap_{i = 1}^r Q_i$. By Theorem \ref{3}, we can have 
$v = x_1^{a_1 -1} \cdots x_n^{a_n -1}$, such that $P = (I \colon v)$. 
Suppose that there exists $v'  = x_1^{b_1} \cdots x_n^{b_n}$, such that 
$v' \neq v$ and $P = (I \colon v')$. Since, $P = (I \colon v')$, by 
Proposition \ref{c}, we must have a $Q' = \langle x_1^{b_1 +1}, \dots, x_n^{b_n +1} \rangle$ in an irredundant irreducible decomposition of $I$. 
We observe that 
$v' \neq v$ implies $ b_j \neq a_j +1$ for some $j \in [n]$, therefore, 
$Q \neq Q'$. Thus, we get two $P$-primary components in an irredundant 
irreducible decomposition of $I$, which contradicts the given condition. 
Hence, $v$ is unique such that $P = (I \colon v)$.
\end{proof}

\section{Macaulay2 Function for Obtaining $v$}
Theorem \ref{4} gives us an explicit algorithm to find the $v$ for 
a monomial ideal $I$ and its associated prime ideal $P$, such that 
$P = (I \colon v)$. We have created the following function in 
\textit{Macaulay2}, which 
can be used to compute a $v$ for any given monomial ideal $I$ and 
an associated prime ideal $P$. This asks for an associated prime ideal $P$ 
and the $P$-primary component in the irredundant irreducible decomposition, 
in order to find the $v$. 
Note that, the input ideal must be  defined 
as a monomial ideal, as the function \texttt{irreducibleDecomposition} 
in \emph{Macaulay2} is defined only for the monomial ideals.
\medskip

Macaulay2, version 1.12.0.1
with packages: ConwayPolynomials, Elimination, IntegralClosure, InverseSystems,
               LLLBases, PrimaryDecomposition, ReesAlgebra, TangentCone
\begin{scriptsize}
\begin{verbatim}
i1 :  v = I -> (a = associatedPrimes I;
          if (#a == 1)
          then (<< "The associated prime ideal of given ideal is "<<ideal(a_0)<<endl ;
          P = ideal(a_0))
          else (<< "The associated prime ideals of given ideal are as follows : "<< endl;
          for i from 0 to (#a -1) do << "P_"<<i<<" = "<<a_i<<" ; "<< endl;
          e = read "Please enter the index number of P for which you want to find v : " ;
          P = ideal(a_(value(e))));
          l = flatten entries gens I;
          k = #l;
          h = flatten for i from 0 to (k -1) list exponents (l_i);
          t = for j from 0 to ((#gens S)-1) list (for i from 0 to (k-1) list h_i_j);
          d = for i from 0 to (#t -1) list (max t_i);
          b = irreducibleDecomposition I;
          p = for i from 0 to (#b -1) list (if radical(b_i) != P then continue; b_i);
          if (#p == 1)
          then (<< "The "<<P<<"-primary component in IID is "<<ideal(p_0)<<endl ;
          q = ideal(p_0))
          else (<< "The "<<P<<"-primary components of in IID are : "<< endl;
          for i from 0 to (#p -1) do << "Q_"<<i<<" = "<<p_i<<" ; "<< endl;
          r = read "Please enter the index number of Q for which you want to find v : " ;
          t = value r;
          q = ideal(p_(t)));
          w = (l = flatten entries gens q;
          k = #l;
          h = flatten for i from 0 to (k -1) list exponents (l_i);
          sum h);
          f = for i from 0 to (#w -1) list (
          if w_i != 0 then ((ring q)_i)^(w_i -1)
          else ((ring q)_i)^(d_i + (random ZZ)));
          product f)

o1 = v

o1 : FunctionClosure

i2 : S = QQ[x_1..x_8]

o2 = S

o2 : PolynomialRing

i3 : I = monomialIdeal(x_1^4,x_2^7,x_3^5,x_1^3*x_4^2,x_2^4*x_4^2,x_3*x_4^2,x_4^5,x_4^2*x_8^2,x_1*x_8^8)

                     4   7   5   3 2   4 2     2   5   2 2     8
o3 = monomialIdeal (x , x , x , x x , x x , x x , x , x x , x x )
                     1   2   3   1 4   2 4   3 4   4   4 8   1 8

o3 : MonomialIdeal of S

i4 : v(I)
The associated prime ideals of given ideal are as follows : 
P_0 = monomialIdeal (x , x , x , x ) ; 
                      1   2   3   4
P_1 = monomialIdeal (x , x , x , x , x ) ; 
                      1   2   3   4   8
Please enter the index number of P for which you want to find v : 0
                                                                   7   5   2
The ideal (x , x , x , x )-primary component in IID is ideal (x , x , x , x )
            1   2   3   4                                      1   2   3   4

      6 4   5 5 2 13
o4 = x x x x x x x
      2 3 4 5 6 7 8

o4 : S

i5 : I : oo

o5 = monomialIdeal (x , x , x , x )
                     1   2   3   4

o5 : MonomialIdeal of S

i6 : v(I)
The associated prime ideals of given ideal are as follows : 
P_0 = monomialIdeal (x , x , x , x ) ; 
                      1   2   3   4
P_1 = monomialIdeal (x , x , x , x , x ) ; 
                      1   2   3   4   8
Please enter the index number of P for which you want to find v : 1
The ideal (x , x , x , x , x )-primary components of in IID are : 
            1   2   3   4   8
                      4   7   5   2   8
Q_0 = monomialIdeal (x , x , x , x , x ) ; 
                      1   2   3   4   8
                      3   4       5   2
Q_1 = monomialIdeal (x , x , x , x , x ) ; 
                      1   2   3   4   8
Please enter the index number of Q for which you want to find v : 1

      2 3 4 2 8
o6 = x x x x x x
      1 2 4 5 7 8

o6 : S

i7 : I : oo

o7 = monomialIdeal (x , x , x , x , x )
                     1   2   3   4   8

o7 : MonomialIdeal of S
\end{verbatim}
\end{scriptsize}

\end{document}